\documentclass[a4paper,10pt]{amsart}
\usepackage{amsmath}
\usepackage{array}
\usepackage{amsthm}
\usepackage{amssymb}

\usepackage[all]{xy}

\theoremstyle{plain}
\newtheorem*{thm*}{Theorem}
\newtheorem*{prop*}{Proposition}
\newtheorem*{rem*}{Remark}

\newtheorem{thm-i}{Theorem}
\newtheorem{prop-i}[thm-i]{Proposition}

\newtheorem{thm}{Theorem}[section]

\newtheorem{defi}[thm]{Definition}
\newtheorem{prop}[thm]{Proposition}
\newtheorem{lm}[thm]{Lemma}

\newtheorem{claim*}{Claim}

\theoremstyle{remark}
\newtheorem{rem}[thm]{Remark}

\newcommand{\Z}{\mathbb Z}

\newcommand{\C}{{\mathcal{C}}}

\newcommand{\F}{{\mathcal{F}}}

\newcommand{\Gr}{{\textbf{Gr}}}
\newcommand{\gr}{{\textbf{gr}}}
\newcommand{\Ab}{{\textbf{Ab}}}
\newcommand{\ab}{{\textbf{ab}}}

\newcommand{\abe}{{\mathfrak{a}}}

\usepackage{multicol}

\makeatletter
  {\end{multicols}\if@restonecol\onecolumn\else\clearpage\fi}
\makeatother
	\newdir{>>}{{}*!/3.5pt/:(1,-.2)@^{>}*!/3.5pt/:(1,+.2)@_{>}*!/7pt/:(1,-.2)@^{>}*!/7pt/:(1,+.2)@_{>}}
\newdir{ >>}{{}*!/8pt/@{|}*!/3.5pt/:(1,-.2)@^{>}*!/3.5pt/:(1,+.2)@_{>}}
\newdir{ |>}{{}*!/-3.5pt/@{|}*!/-8pt/:(1,-.2)@^{>}*!/-8pt/:(1,+.2)@_{>}}
\newdir{ >}{{}*!/-8pt/@{>}}
\newdir{>}{{}*:(1,-.2)@^{>}*:(1,+.2)@_{>}}
\newdir{<}{{}*:(1,+.2)@^{<}*:(1,-.2)@_{<}}

\makeindex

\begin{document}

	\title{Extensions between functors from free groups}

		      \author{Christine Vespa}

\address{ Universit\'e de Strasbourg, Institut de Recherche
  Math\'ematique Avanc\'ee, Strasbourg, France. }                        
\email{vespa@math.unistra.fr}    

\date{\today}

\begin{abstract}
Motivated in part by the study of the stable homology of automorphism groups of free groups, we consider cohomological calculations in the category $\F(\gr)$ of functors from finitely generated free groups to abelian groups.
In particular, we compute the groups $Ext^*_{\F(\gr)}(T^n \circ \abe, T^m \circ \abe)$ where $\abe$ is the abelianization functor and $T^n$ is the n-th tensor power functor for abelian groups. These groups are shown to be non-zero if and only if $*=m-n \geq 0$ and $Ext^{m-n}_{\F(\gr)}(T^n \circ \abe, T^m \circ \abe)=\Z[Surj(m,n)]$ where $Surj(m,n)$ is the set of surjections from a set having $m$ elements to a set having $n$ elements. We make explicit the action of symmetric groups on these groups and the Yoneda and external products. We deduce from these computations those of rational Ext-groups for functors of the form $F \circ \abe$ where $F$ is a symmetric or an exterior power functor. Combining these computations with a recent result of Djament we obtain explicit computations of stable homology  of automorphism groups of free groups with coefficients given by particular contravariant functors. 

\vspace{.3cm}

\end{abstract}
\begin{footnote}{Mathematics Subject Classification (2010): 18G15, 18A25, 20J06
}
\end{footnote}

\maketitle


Stable homology with twisted coefficients of various families of groups can be computed thanks to functor homology in a suitable category (see \cite{FFSS} \cite{Sco} \cite{DV} \cite{DV2}). In particular, stable homology of automorphism groups of free groups with coefficients given by a reduced polynomial covariant functor is trivial (see \cite{DV2}). 
Recently, Djament proved in \cite{D15} that stable homology (resp. cohomology) of automorphism groups of free groups with coefficients given by a reduced polynomial contravariant (resp. covariant) functor is governed by Tor groups (resp. Ext groups) in the category of functors from finitely generated free groups to abelian groups. The aim of this paper is to calculate Ext  and Tor groups between concrete functors from groups to abelian groups in order to obtain explicit computations of stable homology  of automorphism groups of free groups with coefficients given by a contravariant functor. Let $\gr$ be a small skeleton of the category of finitely generated free groups, $\F(\gr)$ the category of functors from $\gr$ to abelian groups and $\abe$ the abelianization functor in $\F(\gr)$.
The main result of this paper is:
\begin{thm-i} \label{thm-intro}

Let $n$ and $m$ be natural integers, we have an isomorphism:
$$Ext^{*}_{\F(\gr)}(T^n \circ \abe, T^m \circ \abe)\simeq  \left\lbrace\begin{array}{ll}
 \Z[Surj(m,n)] & \text{if } *=m-n\\
 0 & \text{otherwise}
 \end{array}
 \right.$$
 where $Surj(m,n)$ is the set of surjections from a set having $m$ elements to a set having $n$ elements. 
 
The actions of the symmetric groups $\mathfrak{S}_m$ and $\mathfrak{S}_n$ on $Ext^{*}_{\F(\gr)}(T^n \circ \abe, T^m \circ \abe)$ are induced by the composition of surjections via the previous isomorphism, up to a sign (see Proposition \ref{action2} for the precise signs).

 The Yoneda product is induced by the composition of surjections, up to a sign, (see Proposition \ref{products} for the precise signs) via the previous isomorphism and the external product is induced by the disjoint union of sets.

\end{thm-i}

We remark that this theorem can be expressed elegantly in terms of symmetric sequences (see Proposition \ref{suites-sym}).

In particular for $n=m=1$ we obtain that $Hom_{\F(\gr)}(\abe, \abe)=\Z$ and $Ext^{*}_{\F(\gr)}(\abe, \abe)=0$ for $*>0$. Let $\ab$ be a small skeleton of the category of finitely generated free abelian groups and $\F(\ab)$ be the category of functors from $\ab$ to abelian groups. The previous groups should be compared with the groups $Ext^{*}_{\F(\ab)}(Id, Id)$ corresponding to the MacLane homology of $\Z$. This homology was computed by B\"okstedt by topological methods (unpublished) and reobtained in \cite{FP} by algebraic methods, in particular it is non trivial. Theorem \ref{thm-intro} illustrates that functor homology in $\F(\gr)$ is easier than  functor homology in $\F(\ab)$ (see also \cite{DPV} for another illustration of this fact).

Theorem \ref{thm-intro} is used in \cite[Proposition 4.1 and 4.6]{DPV} in order to compute homological dimension in a category of polynomial functors related to $\F(\gr)$.

The proof of the first part of Theorem \ref{thm-intro} can be decomposed into two steps corresponding to the first two sections of this paper. In the first we compute $Ext^{*}_{\F(\gr)}(\abe, T^m \circ \abe)$ using an explicit projective resolution of the abelianization functor $\abe$. In the second we deduce our result using the sum-diagonal adjunction and an exponential-type property of the tensor power functor $T^\bullet$. The following section is devoted to the study of products on these Ext-groups and the description of the PROP governing these Ext-groups. More precisely, for $g\textbf{Ab}^-$ the category of graded abelian groups where the commutativity is given by the Koszul sign rule:
 $v \otimes w \mapsto (-1)^{|v|.|w|} w \otimes v$, we deduce from Theorem \ref{thm-intro} the following results:
 
 \begin{thm-i}(Propositions \ref{operad} and \ref{PROP})
 \begin{enumerate}
\item 
The graded symmetric sequence $\mathcal{Q}=\{\mathcal{Q}(n)\}_{n\geq 0}$ given by $\mathcal{Q}(n)=Ext^{*}_{\F(\gr)}(\abe, T^n \circ \abe)$ is an operad in $g\textbf{Ab}^-$.
\item 
The symmetric monoidal graded category $\mathcal{E}$ having as objects the finite sets $n$ and such that 
 $$Hom_{\mathcal{E}}(m,n)=Ext^{*}_{\F(\gr)}(T^n \circ \abe, T^m \circ \abe)$$
 is  isomorphic to the graded PROP freely generated by the operad $\mathcal{Q}$.
\end{enumerate}
 \end{thm-i}

In the final section we deduce from Theorem \ref{thm-intro} rational computations of Ext-groups between symmetric power functors $S^k$, exterior power functors $\Lambda^k$  and tensor power functors. 
In particular, we obtain the following results.
\begin{thm-i} \label{Prop-intro}
Let $n$ and $m$ be natural integers, we have isomorphisms:
$$Ext^{*}_{\F(\gr)}((\Lambda^n \circ \abe) \otimes \mathbb{Q}, (\Lambda^m \circ \abe) \otimes \mathbb{Q}) \simeq \left\lbrace\begin{array}{ll}
 \mathbb{Q}^{\rho(m,n)} & \text{if } *=m-n\\
 0 & \text{otherwise}
 \end{array}
 \right.$$
 where $\rho(m,n)$ denotes the number of partitions of $m$ into $n$  parts.

 $$Ext^{*}_{\F(\gr)}((\Lambda^n \circ \abe) \otimes \mathbb{Q}, (S^m \circ \abe) \otimes \mathbb{Q})\simeq \left\lbrace\begin{array}{ll}
   \mathbb{Q}& \text{if } n=m=0 \text{ and } *=0\\
 & \text{or }  n=m=1 \text{ and } *=0\\
 0 & \text{otherwise}
 \end{array}
 \right.$$

 $$Ext^{*}_{\F(\gr)}((\Lambda^n \circ \abe) \otimes \mathbb{Q}, (T^m \circ \abe) \otimes \mathbb{Q}) \simeq \left\lbrace\begin{array}{ll}
 \mathbb{Q}^{S(m,n)} & \text{if } *=m-n\\
 0 & \text{otherwise}
 \end{array}
 \right.$$
 where $S(m,n)$ denotes the Stirling partition number (i.e. the number of ways to partition a set of $m$ elements into $n$ non-empty subsets).

\end{thm-i}
By duality arguments we deduce similar results for Tor groups.

In \cite{D15}, Djament gives a description of stable homology of automorphisms of free groups with coefficients twisted by contravariant functors in terms of $Tor$-groups between functors in $\F(\gr)$. More precisely, for $N: \gr^{op} \to \mathbb{Q}\text{-Mod}$ a polynomial functor, $F_n$ the free group in $n$ generators and any natural integer $i$, Djament has shown \cite[Th\'eor\`eme $1.10$]{D15} that there is an isomorphism:
$$\underset{n \in \mathbb{N}}{\text{colim}}\ H_i(Aut(F_n), N(F_n)) \simeq \underset{k+l=i}{\bigoplus}{Tor}^\gr_k(N, \Lambda^l \circ(\mathfrak{a}\otimes \mathbb{Q})).$$
This statement is equivalent to the following: for $M: \gr \to \mathbb{Q}\text{-Mod}$ a polynomial functor, 
there is  an isomorphism:
$$\underset{n \in \mathbb{N}}{\text{lim}}\ H^i(Aut(F_n), M(F_n)) \simeq \underset{k+l=i}{\bigoplus}{Ext}_\gr^k(\Lambda^l \circ(\mathfrak{a}\otimes \mathbb{Q}),M).$$

To a covariant functor $M: \gr \to \mathbb{Q}\text{-Mod}$ we can associate a contravariant functor $N: \gr^{op} \to \mathbb{Q}\text{-Mod}$ given by $N(G)=M(Hom_{\Gr}(G, \mathbb{Q}))$ where $\Gr$ is the category of groups.
As an application of Proposition \ref{Prop-intro}, using the previous result of Djament, we obtain the following computations of stable homology:

\begin{thm-i} \label{prop-intro}
Let $H_n=Hom_{\Gr}(F_n, \mathbb{Q})$. We have
$$\underset{n \in \mathbb{N}}{\text{colim}}\ H_*(Aut(F_n), T^d\circ(\mathfrak{a}\otimes \mathbb{Q})(H_n)) \simeq  \left\lbrace\begin{array}{ll}
   \mathbb{Q}^{B(d)}& \text{if } *=d\\
 0 & \text{otherwise}
 \end{array}
 \right.$$
where $B(d)$ denotes the $d$-th Bell number (i.e. the number of partitions of a set of $d$ elements),
$$\underset{n \in \mathbb{N}}{\text{colim}}\ H_*(Aut(F_n), \Lambda^d\circ(\mathfrak{a}\otimes \mathbb{Q})(H_n)) \simeq  \left\lbrace\begin{array}{ll}
   \mathbb{Q}^{\rho(d)}& \text{if } *=d\\
 0 & \text{otherwise}
 \end{array}
 \right.$$
where $\rho(d)$ denotes the number of partitions of $d$, and
$$\underset{n \in \mathbb{N}}{\text{colim}}\ H_*(Aut(F_n), S^d\circ(\mathfrak{a}\otimes \mathbb{Q})(H_n)) \simeq  \left\lbrace\begin{array}{ll}
   \mathbb{Q}& \text{if } *=d=0 \\
    & \text{or\ } *=d=1\\
 0 & \text{otherwise}
 \end{array}
 \right.$$

\end{thm-i}

The two last results of the previous proposition were conjectured by Randal-Williams in \cite{RW}, which motivated the present work. More precisely, combining this proposition with the stable range of such homology groups obtained by Randal-Williams and Wahl in \cite{RWW} we deduce the conjecture of Randal-Williams  in \cite[Corollary 6.4]{RW}. The previous proposition has also been obtained in recent work of Randal-Williams \cite{RW2} using independent topological methods. Proposition \ref{prop-intro} is a partial answer to Problem 17 asked by Morita in \cite{Morita}. 

The final section also contains the following proposition used in the application section of \cite{DPV}.
Let $q_n$ be the left adjoint of the inclusion functor from polynomial functors of degree $\leq n$ to $\F(\gr)$ and $\bar{P}$ in $\F(\gr)$ such that $\Z[\gr(\Z,-)] \simeq \Z \oplus \bar{P}$, we have: 

\begin{prop-i}
Let $m$ and $n$ be integers such that $m\geq n> 0$. We have:
$$Ext^{m-n}_{\F(\gr)}(q_{n}(\bar{P}), T^m \circ \abe) \neq 0.$$
\end{prop-i}

\textbf{Notation}: We denote by $\Gr$ the category of groups, $\Ab$ the category of abelian groups, $\gr$ (resp. $\ab$), a small skeleton of the full subcategory of $\Gr$ (resp. $\Ab$),  having as objects finitely generated free objects. For example, we can take $\gr$ to be the full subcategory with objects the groups $\mathbb{Z}^{* n}$ for $n \in \mathbb{N}$ where $*$ is the free product.

We denote by $\mathfrak{a}$ the abelianization functor from $\gr$ to $\ab$, so that $\abe(\mathbb{Z}^{* n}) =\mathbb{Z}^{\oplus n}$.

Let $\C$ be a small pointed category. We denote by $\F(\C)$ the category of functors from $\C$ to $\Ab$. This category is abelian and has enough projective and injective objects.
A \textit{reduced} functor $F \in \F(\C)$ satisfies $F(0)=0$. 

For $n \in \mathbb{N}$, we denote by $P_n$  the representable functor $P_n:=\Z[\gr(\Z^{*n}, -)]$ in $\F(\gr)$.
By the Yoneda lemma we have a canonical isomorphism $Hom_{\F(\gr)}(P_n, F) \simeq F(\Z^{*n})$.

The Eilenberg-MacLane notion of polynomial functors \cite{EMcL} applies to the category $\F(\gr)$.
For $d \in \mathbb{N}$ we denote by $\F_d(\gr)$ the full subcategory of $\F(\gr)$ of polynomial functors of degree $\leq d$ and $i_d: \F_d(\gr) \to \F(\gr)$ the inclusion functor. This functor has a left adjoint denoted by $q_d$. The functors $q_i(\bar{P})$ are called Passi functors (see \cite{HPV} and \cite{DPV}). For $G \in \gr$ we have a natural isomorphism $q_i(\bar{P})(G) \simeq IG/I^{i+1}G$, where $IG$ is the augmentation ideal (see, for example \cite[Proposition 3.7]{DPV}).

In the following, we sometimes write $Ext^*(F,G)$ instead of $Ext^*_{\F(\gr)}(F,G)$ to shorten formulas.

We denote by  $g\textbf{Ab}^-$ the symmetric monoidal category whose objects are graded abelian groups $V^\bullet=\underset{i \in \mathbb{Z}}{\bigoplus} V^i$ and morphisms are linear maps preserving the grading. The monoidal structure is given by the graded tensor product. The commutativity isomorphism $V^\bullet \otimes W^\bullet \to W^\bullet \otimes V^\bullet$ is defined by:
$$v \otimes w \mapsto (-1)^{|v|.|w|} w \otimes v.$$

\section{Computation of $Ext^{*}_{\F(\gr)}(\abe, T^m \circ \abe)$}

This section is based on the existence of an explicit projective resolution of $\abe$ in $\F(\gr)$. This resolution occurs in \cite[Proposition 5.1]{JP} and plays a crucial r\^ole in \cite{DV2}.

Consider the following simplicial object in $\F(\gr)$:

\begin{equation}\label{simplicial}
\xymatrix{
 \ldots \ar@<+0.7ex>[r]  \ar@{}[r]|-{\ldots}\ar@<-0.7ex>[r] & P_{n+1} \ar@<+0.7ex>[r] \ar@{}[r]|-{\ldots} \ar@<-0.7ex>[r]
&  P_n \ar@<+0.7ex>[r] \ar@{}[r]|-{\ldots}  \ar@<-0.7ex>[r]&  \ldots \ar@<+0.7ex>[r] \ar@<-0.7ex>[r]\ar@<+0.2ex>[r] \ar@<-0.2ex>[r]&  P_2 \ar@<+0.7ex>[r] \ar[r]\ar@<-0.7ex>[r]&  P_1 \ar@<+0.4ex>[r] \ar@<-0.4ex>[r]&  P_0
}
 \end{equation}
 where $\delta_i: P_{n+1} \to P_n$ for $0 \leq i \leq n+1$ are defined by:
$$\delta_0[g_1, g_2, \ldots, g_n, g_{n+1}]=[g_2, \ldots, g_n, g_{n+1}]$$
$$\delta_i[g_1, g_2, \ldots, g_n, g_{n+1}]=[g_1, \ldots, g_ig_{i+1}, \ldots, g_n, g_{n+1}] \text{\quad for} \quad 1 \leq i\leq n$$
$$\delta_{n+1}[g_1, g_2, \ldots, g_n, g_{n+1}]=[g_1, g_2, \ldots, g_n]$$
and $\epsilon_i: P_{n} \to P_{n+1}$ for $1 \leq i \leq n+1$ are defined by:
$$\epsilon_i[g_1, \ldots, g_n]=[g_1, \ldots, g_{i-1}, 1, g_i, \ldots, g_n].$$
We denote by $C_\bullet$ the unnormalized chain complex associated to this simplicial object and $D_\bullet$ the complex defined by $D_i=C_{i+1}$ for $i \geq 0$ and $D_i=0$ for $i<0$.

Recall that the homology of a free group is naturally isomorphic to its abelianization in degree $1$ and is zero in degree $>1$. So $D_\bullet$ is a resolution of $\abe$ and we obtain: 

\begin{prop}(Cf. \cite[Proposition 5.1]{JP})
The exact sequence in $\F(\gr)$:
$$ \ldots P_{n+1} \xrightarrow{d_n} P_n \to \ldots \to P_2 \xrightarrow{d_1} P_1$$
is a projective resolution of the abelianization functor $\abe: \gr \to \Ab$.
The natural transformation $d_n: P_n \to P_{n-1}$ is given on a group $G \in \gr$ by the linear map $\Z[G^{n+1}] \to \Z[G^n]$ such that:
$$d_n([g_1, \ldots, g_{n+1}])=[g_2, \ldots, g_{n+1}]+$$
$$\overset{n}{\underset{i=1}{\sum}}(-1)^{i}[g_1, \ldots, g_{i-1}, g_ig_{i+1}, g_{i+2}, \ldots, g_{n+1}]+(-1)^{n+1}[g_1, \ldots, g_n]$$
for all $(g_1, \ldots, g_{n+1}) \in G^{n+1}$.
\end{prop}

For $F\in \F(\gr)$ and $n$ an integer,
we have $$Ext^n_{\F(\gr)}(\abe, F)\simeq \pi^{n+1}(F(\Z^{*\bullet}))\simeq H^{n+1}(NF(\Z^{*\bullet}))$$
where $F(\Z^{*\bullet})$ is the cosimplicial abelian group obtained by applying $Hom_{\F(\gr)}(-,F)$ to the simplicial object (\ref{simplicial}) and using the natural isomorphism $Hom_{\F(\gr)}(P_n,F) \simeq F(\Z^{*n})$ given by the Yoneda lemma and $NF(\Z^{*\bullet})$ is the normalized cochain complex.
\begin{rem} \label{cas0}
A straightforward calculation gives that
$$Ext^{*}_{\F(\gr)}( \abe, T^0 \circ \abe)\simeq Ext^{*}_{\F(\gr)}( \abe, \mathbb{Z})=0.$$
\end{rem}
\begin{prop} \label{Ext-ab}
Let $m \geq 1$ be a natural integer, we have an isomorphism:
$$Ext^{*}_{\F(\gr)}(\abe, T^m \circ \abe) \simeq \left\lbrace\begin{array}{ll}
 \Z & \text{if } *=m-1\\
 0 & \text{otherwise}
 \end{array}
 \right.$$
 \end{prop}
 The proof of this proposition is based on the following lemma.
 
 \begin{lm} \label{lm-simplicial}
 If $F: \gr \to \Ab$ and $G: \gr \to \Ab$ are reduced functors, then there is a graded morphism
 $$Ext^*_{\F(\gr)}(\abe,F) \otimes Ext^*_{\F(\gr)}(\abe, G) \to Ext^{*+1}_{\F(\gr)}(\abe, F \otimes G)
$$
 which is an isomorphism if the values of $F$ and $Ext^*_{\F(\gr)}(\abe, F)$ are torsion free, 
 \end{lm}
 \begin{proof}
 Considering the complexes $NF(\Z^{*\bullet})$, $NG(\Z^{*\bullet})$ and $NF(\Z^{*\bullet}) \otimes NG(\Z^{*\bullet})$ we have a graded morphism:
 \begin{equation} \label{morphisme}
 H^*(NF(\Z^{*\bullet})) \otimes H^*(NG(\Z^{*\bullet})) \to H^*(NF(\Z^{*\bullet}) \otimes NG(\Z^{*\bullet})).
 \end{equation}

 Let $(F \otimes G)(\Z^{*\bullet})$ be the cosimplicial abelian group obtained when we apply $Hom_{\F(\gr)}(-,F\otimes G)$ to the simplicial object (\ref{simplicial}). We have
 $$Hom_{\F(\gr)}(P_n, F \otimes G) \simeq (F \otimes G)(\Z^{*n})=F(\Z^{*n}) \otimes G(\Z^{*n}).$$
 So $(F \otimes G)(\Z^{*\bullet})=F(\Z^{*\bullet}) \otimes G(\Z^{*\bullet})$ where $\otimes$ denote the tensor product of cosimplicial abelian groups. By the Eilenberg-Zilber theorem, the map
 $$N(F(\Z^{*\bullet}) \otimes G(\Z^{*\bullet})) \to NF(\Z^{*\bullet}) \otimes NG(\Z^{*\bullet})$$
 is a quasi-isomorphism. So, we have isomorphisms
 $$Ext^{n}(\abe, F \otimes G) \simeq H^{n+1}\big(N(F(\Z^{*\bullet}) \otimes G(\Z^{*\bullet}))\big) \simeq H^{n+1}(NF(\Z^{*\bullet}) \otimes NG(\Z^{*\bullet}))$$
 and hence the stated morphisms.
  Suppose now that  the values of $F$ are torsion free, by the K\"unneth formula (see \cite{Weibel} Theorem 3.6.3), the graded morphism (\ref{morphisme}) is part of a natural short exact sequence
   {\footnotesize{$$0 \to H^*(NF(\Z^{*\bullet})) \otimes H^*(NG(\Z^{*\bullet})) \to H^*(NF(\Z^{*\bullet}) \otimes NG(\Z^{*\bullet})) \to Tor_{\Z}^1(Ext^*(\abe, F), Ext^*(\abe, G)) \to 0$$}}
   Hence, if the groups $Ext^*_{\F(\gr)}(\abe, F)$ are torsion free, we have an isomorphism
   $$H^*(NF(\Z^{*\bullet})) \otimes H^*(NG(\Z^{*\bullet})) \simeq H^*(NF(\Z^{*\bullet}) \otimes NG(\Z^{*\bullet})).$$
   We deduce:
  $$ \begin{array}{l}
   Ext^{n+1}(\abe, F \otimes G) = H^{n+2}(NF(\Z^{*\bullet}) \otimes NG(\Z^{*\bullet})) \simeq \underset{i+j=n+2}{\bigoplus}H^i(NF(\Z^{*\bullet})) \otimes H^j(NG(\Z^{*\bullet}))\\
   =\underset{\begin{array}{c} i+j=n+2 \\ i \neq 0, j \neq 0 \end{array}}{\bigoplus}H^i(NF(\Z^{*\bullet})) \otimes H^j(NG(\Z^{*\bullet}))\quad  \text{\ since $F$ and $G$ are reduced}\\
   =\underset{\begin{array}{c} i+j=n+2 \\ i \neq 0, j \neq 0 \end{array}}{\bigoplus}Ext^{i-1}(\abe, F) \otimes Ext^{j-1}(\abe, G)\\
   =\underset{ p+q=n }{\bigoplus}Ext^{p}(\abe, F) \otimes Ext^{q}(\abe, G).
    \end{array}$$
   \end{proof}

\begin{proof}[Proof of Proposition \ref{Ext-ab}]
We prove the result by induction on $m$. To start the induction, $Ext^{*}_{\F(\gr)}(\abe, \abe)$ is the homology of the complex:
$$\xymatrix{\ldots & \Z^{n+1} \ar[l] & \Z^{n} \ar[l]_-{\delta_n} & \ldots \ar[l]&\Z^{2} \ar[l]& \Z \ar[l]_-{\delta_1}}$$
which is trivial for $*\geq 1$ and is isomorphic to $\Z$ for $*=0$. In particular  $Ext^{*}_{\F(\gr)}(\abe, \abe)$ is torsion free. 
Assume that the statement is true for $m$. Applying  Lemma \ref{lm-simplicial} to $T^m \circ \abe$  and $\abe$ which takes torsion free values, we have a graded isomorphism:
$$Ext^{*+1}_{\F(\gr)}(\abe, T^{m+1} \circ \abe)\simeq Ext^{*}_{\F(\gr)}(\abe, T^{m} \circ \abe) \otimes Ext^{*}_{\F(\gr)}(\abe, \abe) $$
and we obtain the result by the inductive step and the computation of $Ext^{*}_{\F(\gr)}(\abe, \abe)$.

\end{proof}

\begin{rem}Proposition \ref{Ext-ab} is used in \cite{DPV} to prove that the global dimension of the category of polynomial functors of degree $\leq m$ from $\textbf{gr}$ to the category of $\mathbb{Q}$-vector spaces has global dimension $m-1$ (see \cite[Proposition 4.6]{DPV}).
\end{rem}

The symmetric group $\mathfrak{S}_m$ acts on $T^m$ by permuting the factors of the tensor product. This action induces an action of $\mathfrak{S}_m$ on the extension groups $Ext^{m-1}_{\F(\gr)}(\abe, T^m \circ \abe)$ that we make explicit in the following proposition.

\begin{prop} \label{action1}
The action of $\mathfrak{S}_m$ on  $Ext^{m-1}_{\F(\gr)}(\abe, T^m \circ \abe)$ is given by the sign representation.
\end{prop}
The proof of this proposition relies on the following well-known lemma (see \cite[X 62]{Bourbaki}).
\begin{lm} \label{signe}
Let $A$ be a commutative ring, $(C,d)$ and $(C', d')$ be complexes of $A$-modules. The map $\sigma(C,C'): C \underset{A}{\otimes} C' \to C' \underset{A}{\otimes} C$ given by $\sigma(C,C')(x \otimes x')=(-1)^{|x|\ |x'|}x' \otimes x$ is an isomorphism of complexes.
\end{lm}

\begin{proof}[Proof of Proposition \ref{action1}]
Since $Ext^{*}_{\F(\gr)}(\abe, \abe)$ is torsion free by Proposition \ref{Ext-ab} and $\abe$ has torsion free values, applying iteratively Lemma \ref{lm-simplicial} for $F=\abe$ and $G=\abe$ we obtain that the morphism
$$\psi: \overset{m}{\underset{i=1}{\bigotimes}}Ext^{*}_{\F(\gr)}(\abe, \abe) \to Ext^{*+m-1}_{\F(\gr)}(\abe, T^m \circ \abe)$$
is an isomorphism.

Let $\sigma \in \mathfrak{S}_m$, the action of $\sigma$ on $Ext^{m-1}_{\F(\gr)}(\abe, T^m \circ \abe)$ is given by the following composition where $\phi: \overset{m}{\underset{i=1}{\bigotimes}}Ext^{*}_{\F(\gr)}(\abe, \abe)    \to \overset{m}{\underset{i=1}{\bigotimes}}Ext^{*}_{\F(\gr)}(\abe, \abe) $ is given by the action of $\sigma$ on $\overset{m}{\underset{i=1}{\bigotimes}}Ext^{*}_{\F(\gr)}(\abe, \abe) $ 

$$\xymatrix{
Ext^{*+m-1}_{\F(\gr)}(\abe, T^m \circ \abe) \ar[r]^-{\simeq}_-{\psi^{-1}} & \overset{m}{\underset{i=1}{\bigotimes}}Ext^{*}(\abe, \abe) \ar[r]^-{\phi}& \overset{m}{\underset{i=1}{\bigotimes}}Ext^{*}(\abe, \abe) \ar[r]^-{\simeq}_-{\psi}& Ext^{*+m-1}_{\F(\gr)}(\abe, T^m \circ \abe) 
}$$
By Lemma \ref{signe}, we have
$$\phi (x_1 \otimes \ldots \otimes x_n)= \epsilon(\sigma) x_{\sigma^{-1}(1)} \otimes \ldots \otimes x_{\sigma^{-1}(m)}$$
where $x_k$ is a generator of the $k$-th copy of $Ext^{*}_{\F(\gr)}(\abe, \abe)$ in $\overset{m}{\underset{i=1}{\bigotimes}}Ext^{*}_{\F(\gr)}(\abe, \abe)$.
\end{proof}

\begin{rem}
We denote by $\Lambda(m)[m-1]$ the symmetric sequence of objects of $g\textbf{Ab}^-$ which is equal to the sign representation of $\mathfrak{S}_m$ placed in degree $m-1$ and $0$ in other degrees. It follows from Propositions \ref{Ext-ab} and \ref{action1} that $Ext^{*}_{\F(\gr)}(\abe, T^m \circ \abe)=\Lambda(m)[m-1]$.
\end{rem}

\section{Computation of $Ext^{*}_{\F(\gr)}(T^n \circ \abe, T^m \circ \abe)$}
This section is based on the sum-diagonal adjunction, an exponential type property of tensor powers and the K\"unneth formula. 

We begin this section with some recollections; we refer the reader to \cite[Appendice B]{DV} for more details.

For $F$ and $G$ in $\F(\gr)$, their external tensor product $F \boxtimes G$ is the functor sending $(X,Y)$ to $F(X) \otimes G(Y)$. This yields a functor:
$$-\boxtimes -:  \F(\gr) \times  \F(\gr) \to  \F(\gr \times \gr).$$

We denote by $\pi_d: \gr ^{\times d} \to \gr$ the functor obtained by iteration of the free product (which is the categorical sum in $\gr$) and $\delta_d: \gr \to \gr^{\times d}$ the diagonal functor. The functor $\delta_d$ is right adjoint of the functor $\pi_d$. We deduce that the functor given by precomposition  $\delta_d^*: \F(\gr^{ \times d}) \to \F(\gr)$ is left adjoint of the functor given by precomposition $\pi_d^*: \F(\gr) \to \F(\gr^{\times d})$. 

A graded exponential functor of $\F(\gr)$ is a sequence $E^\bullet=(E^n)_{n\in \mathbb{N}}$ of objects of $\F(\gr)$ taking finite dimensional values together with natural isomorphisms $E^0 \simeq \Z$ and $\pi_2^*(E^m) \simeq \underset{i+j=m}{\bigoplus}E^{i}\boxtimes E^{j}$.
The tensor power (graded) functor $T^\bullet$ is not exponential but we have a similar property which is useful in the computations below:
\begin{equation} \label{exponential}
\pi_n^*(T^m \circ \abe) \simeq \underset{i_1+ \ldots+i_n=m}{\bigoplus}(T^{i_1} \circ \abe \boxtimes \ldots \boxtimes T^{i_n} \circ \abe) \underset{\mathfrak{S}_{i_1}\times \ldots \times \mathfrak{S}_{i_n}}{\otimes}\Z[\mathfrak{S}_m]
\end{equation}
where $T^0 \circ \abe \simeq \mathbb{Z}$.

The following lemmas will be useful below.
\begin{lm} \label{map-Omega}
Let $Surj(m,n)$ be the set of surjections from the set having $m$ elements to the one having $n$ elements. 
For $f \in Surj(m,n)$ such that for all $k \in \{1, \ldots, n\}$, $|f^{-1}(k) |=i_k$ there are unique maps $\alpha \in Surj(n,n)$ and $s \in Surj(m,n)$ such that $f=s \circ \alpha$ where $\alpha$ is the inverse of a $(i_1, \ldots, i_n)$-shuffle and $s$ is an order preserving surjection.
\end{lm}
\begin{lm}\label{K2}
For functors $F \in \F(\gr)$ and $G\in \F(\gr)$ there is a graded morphism
 $$Ext^*_{\F(\gr)}(\abe,F) \otimes Ext^*_{\F(\gr)}(\abe, G) \to Ext^{*}_{\F(\gr\times \gr)}(\abe \boxtimes \abe, F \boxtimes G)$$
 which is an isomorphism if the values of $F$ and $Ext^*(\abe, F)$ are torsion free.
\end{lm}
\begin{proof}
There is a canonical isomorphism $P_A \boxtimes P_B \simeq \Z[\gr \times \gr \big( (A,B), - \big)]$ so 
$$Hom_{\F(\gr \times \gr)}(P_A\boxtimes P_B, F \boxtimes G) \simeq F(A) \otimes G(B)$$
by the Yoneda lemma.
Then the statement is a consequence of the K\"unneth formula.
\end{proof}

Our first main result is the following theorem.
\begin{thm} \label{Ext}
Let $n$ and $m$ be natural integers, we have an isomorphism:
$$Ext^{*}_{\F(\gr)}(T^n \circ \abe, T^m \circ \abe) \simeq \left\lbrace\begin{array}{ll}
 \Z[Surj(m,n)] & \text{if } *=m-n\\
 0 & \text{otherwise}
 \end{array}
 \right.$$
  where $Surj(m,n)$ is the set of surjections from a set having $m$ elements to a set having $n$ elements. 
\end{thm}
\begin{proof}
\begin{eqnarray*}
&\ &Ext^*_{\F(\gr)}(T^n \circ \abe, T^m \circ \abe) \\
&\simeq& Ext^*_{\F(\gr \times \ldots\times \gr)}( \abe^{\boxtimes n}, T^m \circ \abe \circ \pi_n) \text{\ by the sum-diagonal adjunction}\\
&\simeq &Ext^*_{\F(\gr \times \ldots\times \gr)}\big( \abe^{\boxtimes n}, \underset{i_1+ \ldots+i_n=m}{\bigoplus} (T^{i_1}\circ \abe \boxtimes \ldots \boxtimes T^{i_n}\circ \abe)\underset{\mathfrak{S}_{i_1}\times \ldots \times \mathfrak{S}_{i_n}}{\otimes}\Z[\mathfrak{S}_m] \big)
 \text{\ by \ref{exponential}}\\
&\simeq &\underset{i_1+ \ldots+i_n=m}{\bigoplus}Ext^*_{\F(\gr \times \ldots\times \gr)}\big( \abe^{\boxtimes n}, (T^{i_1}\circ \abe \boxtimes \ldots \boxtimes T^{i_n}\circ \abe)\underset{\mathfrak{S}_{i_1}\times \ldots \times \mathfrak{S}_{i_n}}{\otimes}\Z[\mathfrak{S}_m] \big)\\
&\simeq &\underset{i_1+ \ldots+i_n=m}{\bigoplus}\big(\underset{k=1}{\overset{n}{\bigotimes}}\ Ext^*_{\F(\gr)}( \abe, T^{i_k}\circ \abe )\big)\underset{\mathfrak{S}_{i_1}\times \ldots \times \mathfrak{S}_{i_n}}{\otimes}\Z[\mathfrak{S}_m] \text{\ by Lemma \ref{K2}}\\
&\simeq& \left\lbrace\begin{array}{ll}
\underset{i_1+ \ldots+i_n=m \atop i_k > 0, \forall k }{\bigoplus}\big( \Z \underset{\mathfrak{S}_{i_1}\times \ldots \times \mathfrak{S}_{i_n}}{\otimes}\Z[\mathfrak{S}_m] \big)& \text{if } *=m-n\\
 0 & \text{otherwise}
\end{array}
 \right.
\end{eqnarray*}
where the last isomorphism follows from Proposition \ref{Ext-ab} and Remark \ref{cas0}.

Let $f \in Surj(m,n)$, by Lemma \ref{map-Omega} $f$ admits a unique decomposition of the form $f=s \circ \alpha$ where $\alpha$ is the inverse of a $(|f^{-1}(1)|, \ldots, |f^{-1}(n)|)$-shuffle and $s$ is an order preserving surjection. So we deduce that the map:
$$\kappa: \Z[Surj(m,n)] \to \underset{i_1+ \ldots+i_n=m \atop i_k > 0, \forall k }{\bigoplus}\big( \Z \underset{\mathfrak{S}_{i_1}\times \ldots \times \mathfrak{S}_{i_n}}{\otimes}\Z[\mathfrak{S}_m] \big)$$
given by $\kappa([f])=1 \otimes \alpha \in \Z \underset{\mathfrak{S}_{|f^{-1}(1)|}\times \ldots \times \mathfrak{S}_{|f^{-1}(n)|}}{\otimes}\Z[\mathfrak{S}_m] $
is an isomorphism.
\end{proof}

\begin{rem}
 The previous theorem is used  in \cite{DPV} to prove that, if $n \leq m$, the homological dimension of the functor $T^n \circ \mathfrak{a} \in \F_m(\gr)$ in the category of polynomial functors of degree $\leq m$ from $\textbf{gr}$ to $\textbf{Ab}$ is $m-n$ (see  \cite[Proposition 4.1]{DPV}).
\end{rem}
For $A_n$ and $B_n$ ordered sets of cardinality $n$ and $f: A_n \to B_n$ a bijection, denote by $\bar{f}: B_n \to B_n$ the unique permutation of $B_n$ such that $\bar{f} \circ u =f$ where $u: A_n \to B_n$ is the unique order preserving map.

In the following proposition we make explicit the actions of the symmetric groups on $Ext^{*}_{\F(\gr)}(T^n \circ \abe, T^m \circ \abe)$.

\begin{prop}\label{action2}
The symmetric groups $\mathfrak{S}_m$ and $\mathfrak{S}_n$ act on $Ext^{m-n}_{\F(\gr)}(T^n \circ \abe, T^m \circ \abe)\simeq \Z[Surj(m,n)]$ in the following way: for $\sigma \in \mathfrak{S}_m$, $\tau_{k,l} \in \mathfrak{S}_n$ the transposition of $k$ and $l$ where $k,l \in \{1, \ldots, n\}$ and $f \in Surj(m,n)$
$$[f].\sigma= \underset{1\leq i \leq n}{\prod}\epsilon(\overline{\sigma_{\mid (f \circ \sigma)^{-1}(i)}})[f\circ \sigma]$$
where $\sigma_{\mid (f \circ \sigma)^{-1}(i)}:(f \circ \sigma)^{-1}(i)\to \sigma((f \circ \sigma)^{-1}(i))$ $$\tau_{k,l}.[f]=(-1)^{(|f^{-1}(k)|-1)(|f^{-1}(l)|-1)} [\tau_{k,l} \circ f].$$
\end{prop}

\begin{proof}
Let $f$ be an element of $Surj(m,n)$ such that $f=s\circ \sigma$ is the unique decomposition of $f$ of the form described in Lemma \ref{map-Omega}.

For $\sigma \in \mathfrak{S}_m$, the action of $\sigma$ on $\Z[Surj(m,n)]$ is given by the following composition, where $\phi$ is the map given by the action of $\sigma$ on $$ \underset{i_1+ \ldots +i_n=m}{\bigoplus} (\underset{k=1}{\overset{n}{\otimes}} Ext^{i_k-1}(\abe, \abe^{\otimes i_k}) \underset{\mathfrak{S}_{i_1} \times \ldots \times \mathfrak{S}_{i_n}}{\bigotimes} \Z[\mathfrak{S}_m]) $$
{\footnotesize{
$$\xymatrix{
\Z[Surj(m,n)]\ar[r]^-{\simeq}_-{\kappa} & \underset{i_1+ \ldots +i_n=m}{\bigoplus} \Z \underset{\mathfrak{S}_{i_1} \times \ldots \times \mathfrak{S}_{i_n}}{\bigotimes} \Z[\mathfrak{S}_m] \ar[d]^{\simeq}\\
&   \underset{i_1+ \ldots +i_n=m}{\bigoplus} (\underset{k=1}{\overset{n}{\otimes}} Ext^{i_k-1}(\abe, \abe^{\otimes i_k}) \underset{\mathfrak{S}_{i_1} \times \ldots \times \mathfrak{S}_{i_n}}{\bigotimes} \Z[\mathfrak{S}_m]) \ar[d]^\phi \\
&  \underset{i_1+ \ldots +i_n=m}{\bigoplus} (\underset{k=1}{\overset{n}{\otimes}} Ext^{i_k-1}(\abe, \abe^{\otimes i_k}) \underset{\mathfrak{S}_{i_1} \times \ldots \times \mathfrak{S}_{i_n}}{\bigotimes} \Z[\mathfrak{S}_m] )  \ar[d]^{\simeq}\\
 \Z[Surj(m,n)]  & \underset{i_1+ \ldots +i_n=m}{\bigoplus} \Z \underset{\mathfrak{S}_{i_1} \times \ldots \times \mathfrak{S}_{i_n}}{\bigotimes} \Z[\mathfrak{S}_m] \ar[l]^-{\kappa^{-1}} 
}$$}}
we have
$$\kappa^{-1} \circ \phi \circ \kappa ([f])=\kappa^{-1} \circ \phi \circ \kappa ([s\circ \alpha])
=\kappa^{-1} \circ \phi \left((1 \otimes [\alpha])_{|f^{-1}(1)|, \ldots,|f^{-1}(n)|} \right)
$$
$$=\kappa^{-1}(\underset{1\leq i \leq n}{\prod}\epsilon(\overline{\sigma_{\mid (f \circ \sigma)^{-1}(i)}}) \ (1 \otimes [\alpha \circ \sigma])_{|f^{-1}(1)|, \ldots,|f^{-1}(n)|})=\underset{1\leq i \leq n}{\prod}\epsilon(\overline{\sigma_{\mid (f \circ \sigma)^{-1}(i)}}) [s \circ \alpha \circ \sigma]$$
$$=\underset{1\leq i \leq n}{\prod}\epsilon(\overline{\sigma_{\mid (f \circ \sigma)^{-1}(i)}}) [f \circ \sigma]$$
where we deduce the third equality from Proposition \ref{action1} and where $(-)_{i_1, \ldots, i_n}$ corresponds to the summand indexed by $i_1, \ldots, i_n$.

For $\tau_{k,l} \in \mathfrak{S}_n$, we have $\tau_{k,l} \circ f=\tau_{k,l} \circ s \circ \alpha=s' \circ T_{k,l} \circ \alpha$ where $s'$ is the order preserving surjection from $m$ to $n$ such that $|s'^{-1}(l)|=|s^{-1}(k)|$, $|s'^{-1}(k)|=|s^{-1}(l)|$ and $|s'^{-1}(q)|=|s'^{-1}(q)|$ for $q \in \{1, \ldots, m\} \setminus \{k,l\}$ and $T_{k,l}$ is the permutation obtained by the bloc transposition of $s^{-1}(k)$ and $s^{-1}(l)$. Note that $T_{k,l} \circ \alpha$ is the inverse of a $(|(\tau_{k,l}\circ f)^{-1}(1)|, \ldots,|(\tau_{k,l}\circ f)^{-1}(n)|)$-shuffle.

For $\tau_{k,l}  \in \mathfrak{S}_n$, the action of $\tau_{k,l} $ on $\Z[Surj(m,n)]$ is given by the following composition, where $\phi'$ is the map given by the action of $\tau_{k,l} $ on  $$ \underset{i_1+ \ldots +i_n=m}{\bigoplus} (\underset{k=1}{\overset{n}{\otimes}} Ext^{i_k-1}(\abe, \abe^{\otimes i_k}) \underset{\mathfrak{S}_{i_1} \times \ldots \times \mathfrak{S}_{i_n}}{\bigotimes} \Z[\mathfrak{S}_m]) $$
{\footnotesize{
$$\xymatrix{
\Z[Surj(m,n)]\ar[r]^-{\simeq}_-{\kappa} & \underset{i_1+ \ldots +i_n=m}{\bigoplus} \Z \underset{\mathfrak{S}_{i_1} \times \ldots \times \mathfrak{S}_{i_n}}{\bigotimes} \Z[\mathfrak{S}_m] \ar[d]^{\simeq}\\
&   \underset{i_1+ \ldots +i_n=m}{\bigoplus} (\underset{k=1}{\overset{n}{\otimes}} Ext^{i_k-1}(\abe, \abe^{\otimes i_k}) \underset{\mathfrak{S}_{i_1} \times \ldots \times \mathfrak{S}_{i_n}}{\bigotimes} \Z[\mathfrak{S}_m]) \ar[d]^{\phi'} \\
&  \underset{i_1+ \ldots +i_n=m}{\bigoplus} (\underset{k=1}{\overset{n}{\otimes}} Ext^{i_k-1}(\abe, \abe^{\otimes i_k}) \underset{\mathfrak{S}_{i_1} \times \ldots \times \mathfrak{S}_{i_n}}{\bigotimes} \Z[\mathfrak{S}_m] )  \ar[d]^{\simeq}\\
 \Z[Surj(m,n)]  & \underset{i_1+ \ldots +i_n=m}{\bigoplus} \Z \underset{\mathfrak{S}_{i_1} \times \ldots \times \mathfrak{S}_{i_n}}{\bigotimes} \Z[\mathfrak{S}_m] \ar[l]^-{\kappa^{-1}} 
}$$}}

we have
$$\kappa^{-1} \circ \phi' \circ \kappa ([f])=\kappa^{-1} \circ \phi' \circ \kappa ([s\circ \alpha])
=\kappa^{-1} \circ \phi' \left((1 \otimes [\alpha])_{|f^{-1}(1)|, \ldots,|f^{-1}(n)|} \right)
$$
$$=\kappa^{-1}((-1)^{(i_k-1)(i_l-1)} 1 \otimes [T_{k,l} \circ \alpha]_{|(\tau_{k,l}\circ f)^{-1}(1)|, \ldots,|(\tau_{k,l}\circ f)^{-1}(n)|})
$$
$$=(-1)^{(i_k-1)(i_l-1)}  [s' \circ \alpha']=(-1)^{(i_k-1)(i_l-1)} [\tau_{k,l}\circ f]$$
where the third equality is a consequence of Lemma \ref{signe}.
\end{proof}

\section{Products and the structure of PROP}
The aim of this section is to make explicit the Yoneda product and the external product on $Ext^*_{\F(\gr)}(T^n \circ \abe, T^m \circ \abe)$ and to deduce the structure of the PROP governing the graded groups $Ext^*_{\F(\gr)}(T^n \circ \abe, T^m \circ \abe)$.
\begin{prop} \label{products}
\begin{enumerate}
\item The Yoneda product:
$$\mathcal{Y}: Ext^{m-l}_{\F(\gr)}(T^l \circ \abe, T^m \circ \abe) \otimes Ext^{n-m}_{\F(\gr)}(T^m \circ \abe, T^n \circ \abe) \to Ext^{n-l}_{\F(\gr)}(T^l \circ \abe, T^n \circ \abe)$$
is induced, via the isomorphism of Theorem \ref{Ext},  by the map 
$$Y: \Z[Surj(m,l)] \otimes \Z[Surj(n,m)] \to \Z[Surj(n,l)]$$
given by 
$$Y([g] \otimes [f])=\underset{k=1}{\overset{l}{\prod}} \epsilon(\overline{{\tau_{i,j}}_{|g^{-1}(k)}} )  (-1)^{(|f^{-1}(i)|-1)(|f^{-1}(j)|-1)}[g\circ f]$$
where $f \in Surj(n,m)$ and $g \in Surj(m,l)$ is of the form $g=s \circ \tau_{i, j}$ where $s$ is an order preserving surjection and $\tau_{i, j} \in \mathfrak{S}_m$ denotes the transposition of $i$ and $j$.

\item 
The external product
$$Ext^{m-l}_{\F(\gr)}(T^l \circ \abe, T^m \circ \abe) \otimes Ext^{n-p}_{\F(\gr)}(T^p \circ \abe, T^n \circ \abe) \to Ext^{m+n-l-p}_{\F(\gr)}(T^{l+p} \circ \abe, T^{m+n} \circ \abe)$$
is induced by the disjoint union of sets via the isomorphism of  Theorem \ref{Ext}.
\end{enumerate}
\end{prop}

\begin{proof}
\begin{enumerate}
\item
Considering the action of $Hom_{\F(\gr)}(T^m \circ \abe, T^m \circ \abe) \simeq \Z[\mathfrak{S}_m]$ on $Ext^{m-l}_{\F(\gr)}(T^l \circ \abe, T^m \circ \abe) $ and on $Ext^{n-m}_{\F(\gr)}(T^m \circ \abe, T^n \circ \abe)$ we obtain that the map 
$$Ext^{m-l}_{\F(\gr)}(T^l \circ \abe, T^m \circ \abe) \otimes Ext^{n-m}_{\F(\gr)}(T^m \circ \abe, T^n \circ \abe) \to Ext^{n-l}_{\F(\gr)}(T^l \circ \abe, T^n \circ \abe)$$
induces a map:
$$Ext^{m-l}(T^l \circ \abe, T^m \circ \abe) \underset{\mathfrak{S}_m}{\otimes} Ext^{n-m}(T^m \circ \abe, T^n \circ \abe) \to Ext^{n-l}(T^l \circ \abe, T^n \circ \abe).$$
So the signs can easily be deduced from Proposition \ref{action2}.
\item
We have 
$$Ext^{m-l}(T^l \circ \abe, T^m \circ \abe) \simeq  \underset{i_1+ \ldots +i_l=m}{\bigoplus} (\underset{k=1}{\overset{l}{\otimes}} Ext^{i_l-1}(\abe, \abe^{\otimes i_l}) \underset{\mathfrak{S}_{i_1} \times \ldots \times \mathfrak{S}_{i_l}}{\bigotimes} \Z[\mathfrak{S}_m]) $$
$$Ext^{n-p}(T^p \circ \abe, T^n \circ \abe) \simeq  \underset{j_1+ \ldots +j_p=n}{\bigoplus} (\underset{k=1}{\overset{p}{\otimes}} Ext^{j_p-1}(\abe, \abe^{\otimes j_p}) \underset{\mathfrak{S}_{j_1} \times \ldots \times \mathfrak{S}_{j_p}}{\bigotimes} \Z[\mathfrak{S}_n]) $$
and
$$Ext^{m+n-l-p}(T^{l+p} \circ \abe, T^{m+n} \circ \abe) \simeq  \underset{i_1+ \ldots +i_{l+p}=m+n}{\bigoplus} (\underset{k=1}{\overset{l+p}{\otimes}} Ext^{i_{l+p}-1}(\abe, \abe^{\otimes i_{l+p}}) \underset{\mathfrak{S}_{i_1} \times \ldots \times \mathfrak{S}_{i_{l+p}}}{\bigotimes} \Z[\mathfrak{S}_{m+n}]) $$
For fixed $i_1, \ldots, i_l, j_1, \ldots j_k$ such that $i_1+ \ldots +i_l=m$ and $j_1+ \ldots +j_p=n$, $E_k \in Ext^{i_k-1}(\abe, \abe^{\otimes i_k})$ and $F_\alpha \in Ext^{j_\alpha-1}(\abe, \abe^{\otimes j_\alpha})$ and $\sigma \in \mathfrak{S}_n$ and $\sigma' \in \mathfrak{S}_m$ the external product is given by the map
$$E_1 \otimes \ldots E_l \otimes \sigma \otimes F_1 \otimes \ldots \otimes F_p \otimes \sigma' \mapsto E_1 \otimes \ldots \otimes E_l \otimes F_1 \otimes \ldots \otimes F_p \otimes \sigma \vee \sigma'$$
where $\sigma \vee \sigma': m+n \to m+n$ is the permutation such that $\sigma \vee \sigma'(i)=\sigma(i)$ for $1\leq i\leq n$ and $\sigma \vee \sigma'(i)=\sigma'(i)$ for $n+1\leq i\leq n+m$. This map corresponds to the disjoint union via the isomorphism obtained in Theorem \ref{Ext}.
\end{enumerate}
\end{proof}
For $i_1, \ldots, i_k$ integers, iterating the external product, we obtain a graded map
$$\mathcal{E}^{i_1, \ldots,  i_k}: Ext^{*}_{\F(\gr)}(\abe, T^{i_1} \circ \abe) \otimes \ldots \otimes Ext^{*}_{\F(\gr)}(\abe, T^{i_k} \circ \abe) \to Ext^{*}_{\F(\gr)}(T^{k} \circ \abe, T^{i_1+ \ldots+ i_k} \circ \abe).$$
This map is $\mathfrak{S}_{i_1} \times \ldots \times \mathfrak{S}_{i_k}$-equivariant by Propositions \ref{action1} and \ref{action2}. It induces a map
$$\underset{i_1+ \ldots+ i_k=m \atop i_p>0, \forall p}{\bigoplus}\overset{k}{\underset{p=1}{\bigotimes}}Ext^{*}_{\F(\gr)}(\abe, T^{i_p} \circ \abe)\underset{\mathfrak{S}_{i_1} \times \ldots \times \mathfrak{S}_{i_k}}{\otimes} \mathbb{Z}[\mathfrak{S}_{m}]\to Ext^{*}_{\F(\gr)}(T^{k} \circ \abe, T^{m} \circ \abe)$$
which is an isomorphism by the exponential type property (see the proof of Theorem \ref{Ext}). We deduce that the elements of $Ext^{*}_{\F(\gr)}(T^{k} \circ \abe, T^{m} \circ \abe)$ are obtained from elements of $Ext^{*}_{\F(\gr)}(\abe, T^{i} \circ \abe)$ using the external product. In other words, the isomorphism of Theorem \ref{Ext} corresponds to the following statement.

\begin{prop} \label{suites-sym}
The external product induces, with the respect to the action on the second variable, an isomorphism
$$\mathbb{T}(\underset{m\geq 0}{\bigoplus} \Lambda(m)[m-1])\xrightarrow{\simeq} Ext(T^{\bullet_1} \circ \abe, T^{\bullet_2} \circ \abe)$$
where $\mathbb{T}$ is the tensor product of symmetric sequences and $\bullet_1$ corresponds to the tensor length. The action of the symmetric group $\mathfrak{S}_{\bullet_1}$ corresponds to the usual signs for the grading $[-]$.
\end{prop} 

Recall (see for example \cite[section 2.1]{Fresse}) that tensor product of symmetric sequences $M$ and $N$ is given by 
$$(M \otimes N)(n)=\underset{p+q=n}{\bigoplus} Ind_{\mathfrak{S}_p \times \mathfrak{S}_q}^{\mathfrak{S}_n}M(p) \otimes N(q).$$

By Proposition \ref{products} $(2)$, the graded category $\mathcal{E}$ having as objects natural numbers, where the morphisms from $m$ to $l$ are given by the groups $Ext^{*}_{\F(\gr)}(T^l \circ \abe, T^m \circ \abe)$ and where the composition is given by the Yoneda product is a graded PROP (i.e. a graded symmetric monoidal category with objects the natural numbers whose symmetric monoidal structure is given by the sum of integers). We prove below that this graded PROP is the PROP associated to its endomorphism operad whose definition we recall.
\begin{defi} \label{operad}
The endomorphism operad associated to the graded PROP $\mathcal{E}$ is the graded symmetric sequence
$\mathcal{Q}=\{\mathcal{Q}(n)\}_{n\geq 0}$ given by $\mathcal{Q}(n)=Ext^{*}_{\F(\gr)}(\abe, T^n \circ \abe)$. The maps
$$\gamma(i_1, \ldots, i_k): \mathcal{Q}(k) \otimes \mathcal{Q}(i_1) \otimes \ldots \mathcal{Q}(i_k) \to \mathcal{Q}(i_1+ \ldots +i_k)$$
are given by 
$$\gamma(i_1, \ldots, i_k)(x, x_{i_1}, \ldots, x_{i_k})=\mathcal{Y}(x, \mathcal{E}^{i_1, \ldots,  i_k}(x_{i_1}, \ldots, x_{i_k}))$$
for $x \in \mathcal{Q}(k)$ and $x_{i_l}\in \mathcal{Q}(i_l)$.
\end{defi}

\begin{rem}
Note the similarity between the operad $\mathcal{Q}$ and the \textit{determinant operad}  introduced by Ginzburg and Kapranov in \cite[1.3.21]{GK}. The \textit{determinant operad} $\mathcal{P}$  is the operad in $g\textbf{Ab}^-$  such that $\mathcal{P}(n)= \Lambda(n)[1-n]$ (the sign representation of $\mathfrak{S}_n$ placed in degree $(1-n)$).
\end{rem}

The graded PROP $\mathcal{C}$ freely generated by a graded operad $P$ is a graded symmetric monoidal category such that $\C(1,n)=P(n)$. We refer the reader to \cite[Definition 1.5]{HPV} for the concrete description of the composition in this PROP in the ungraded case, which can be easily extend to the graded case. Note that there are PROPs which do not come from an operad. 

\begin{prop}\label{PROP}
The graded category $\mathcal{E}$ is isomorphic to the graded PROP $\Omega(\mathcal{Q})$ freely generated by the operad  $\mathcal{Q}$.
\end{prop}
\begin{proof}
By definition of the graded PROP freely generated by an operad we have a canonical functor
$$\Omega(\mathcal{Q}) \to \mathcal{E}$$
which is bijective on objects and on morphisms by Theorem \ref{Ext}. So it is an isomorphism of categories.
\end{proof}

\section{Applications}
In this section, we deduce from Theorem \ref{Ext} information on Ext-groups between some other functors of $\F(\gr)$.

The next result is used in the proof  \cite[corollaire 3.3]{DPV}.

\begin{prop}
Let $m$ and $n$ be integers such that $m \geq n > 0$. We have:
$$Ext^{m-n}_{\F(\gr)}(q_{n}(\bar{P}), T^m \circ \abe) \neq 0.$$
\end{prop}
\begin{proof}
Let $g \in Hom(\abe^{\otimes n}, q_n \bar{P})$ be the kernel of the map $q_n \bar{P} \to q_{n-1} \bar{P}$. More concretely, for $G \in \gr$, $g: \abe^{\otimes n}(G) \to q_n \bar{P}(G) \simeq IG/I^{n+1}G$ is given by 
$$g(g_1 \otimes \ldots \otimes g_n)=([g_1]-1) \ldots ([g_n]-1).$$

We denote by $f$ the canonical generator of $Hom_{\F_n(\gr)}(q_n \bar{P}, \abe^{\otimes n}) \simeq \Z$ i.e. the natural transformation corresponding to $1 \in \Z$ by the natural isomorphisms
$$Hom_{\F_n(\gr)}(q_n \bar{P}, \abe^{\otimes n})  \simeq Hom_{\F(\gr)}(\bar{P}, \abe^{\otimes n})  \simeq cr_1(\abe^{\otimes n})(\Z) \simeq \Z$$
where the second isomorphism is given by the Yoneda lemma and where $cr_1$ denotes the first cross-effect. More explicitly, for $G \in \gr$, the group morphism $f: IG/I^{n+1}G\simeq q_n \bar{P}(G) \to  \abe^{\otimes n}(G)$ is given by:
$$f([g]-1)=g^{\otimes n}.$$
Using the relation $([g_1]-1)([g_2]-1)=([g_1+g_2]-1)-([g_1]-1)-([g_2]-1)$ in $IG$, we obtain
$$f(([g_1]-1)([g_2]-1)\ldots ([g_n]-1))=\sum_{\sigma \in \mathfrak{S}_n} g_{\sigma(1)} \otimes  \ldots \otimes g_{\sigma(n)}.$$

We deduce that the composition $\abe^{\otimes n} \xrightarrow{g} q_n \bar{P} \xrightarrow{f} \abe^{\otimes n}$ is the trace map denoted by $tr$. This gives rise to a commutative diagram:
$$\xymatrix{
Ext^{m-n}_{\F(\gr)}(T^n \circ \abe, T^m \circ \abe) \ar[d] \ar[r]^{tr^*}&Ext^{m-n}_{\F(\gr)}(T^n \circ \abe, T^m \circ \abe)\\
Ext^{m-n}_{\F(\gr)}(q_{n}(\bar{P}), T^m \circ \abe). \ar[ur]
}$$
We prove that the map $tr^*$ is non-zero. Let $f \in Surj(m,n)$ be the surjection defined by $f(i)=i$ for $i \in \{1, \ldots, n-1 \}$ and $f(k)=n$ for $k\geq n$. By Proposition \ref{action2} we have: 
$$\overline{tr^*}(f)=\underset{\sigma \in \mathfrak{S}_n}{\sum} \epsilon_\sigma [\sigma \circ f]=[f]+ \underset{\sigma \in \mathfrak{S}_n \setminus \{Id\}}{\sum} \epsilon_\sigma [\sigma \circ f] \neq 0$$
where $\epsilon_\sigma \in \{-1, 1\}$ and $\overline{tr^*}(f): \Z[Surj(m,n)] \to \Z[Surj(m,n)]$ is the map induced by $tr^*$ via the isomorphism in Proposition \ref{Ext}.

So, we deduce from the previous commutative diagram that $Ext^{m-n}_{\F(\gr)}(q_{n}(\bar{P}), T^m \circ \abe)$ is non zero.
\end{proof}

In the rest of this section we deduce from Theorem \ref{Ext} the computation of  Ext-groups between certain functors from $\gr$ to the category of $\mathbb{Q}$-modules. 

For $M$ an abelian group, recall that $S^n(M)=(T^n(M))_{\mathfrak{S}_n}$ where $\mathfrak{S}_n$ acts by the permutation of variables and $\Lambda^n(M) \otimes \Z[1/2]=(T^n(M) \otimes \Z[1/2])_{\mathfrak{S}_n}$ where $\mathfrak{S}_n$ acts by the permutation of variables and the multiplication by the signature.

As $\mathbb{Q}[\mathfrak{S}_n]$ is semi-simple, the functors $(S^n \circ \abe) \otimes \mathbb{Q}$ and $(\Lambda^n \circ \abe) \otimes \mathbb{Q}$ are direct summands of the functor $(T^n \circ \abe) \otimes \mathbb{Q}$. This allows us to obtain the rational computations in  the following theorem.

\begin{thm} \label{Ext-ext}
Let $n$ and $m$ be natural integers, we have isomorphisms:

$$Ext^{*}_{\F(\gr)}((\Lambda^n \circ \abe) \otimes \mathbb{Q}, (\Lambda^m \circ \abe) \otimes \mathbb{Q}) \simeq \left\lbrace\begin{array}{ll}
 \mathbb{Q}^{\rho(m,n)} & \text{if } *=m-n\\
 0 & \text{otherwise}
 \end{array}
 \right.$$
 where $\rho(m,n)$ denotes the number of partitions of $m$ into $n$  parts.
 
 $$Ext^{*}_{\F(\gr)}((S^n \circ \abe) \otimes \mathbb{Q}, (S^m \circ \abe) \otimes \mathbb{Q}) \simeq \left\lbrace\begin{array}{ll}
  \mathbb{Q}& \text{if } n=m \text{ and } *=0\\
   0 & \text{otherwise}
 \end{array}
 \right.$$

 $$Ext^{*}_{\F(\gr)}((\Lambda^n \circ \abe) \otimes \mathbb{Q}, (S^m \circ \abe) \otimes \mathbb{Q})\simeq \left\lbrace\begin{array}{ll}
   \mathbb{Q}& \text{if } n=m=0 \text{ and } *=0\\
 & \text{or }  n=m=1 \text{ and } *=0\\
 0 & \text{otherwise}
 \end{array}
 \right.$$

 $$Ext^{*}_{\F(\gr)}((S^n \circ \abe) \otimes \mathbb{Q}, (\Lambda^m \circ \abe) \otimes \mathbb{Q}) \simeq \left\lbrace\begin{array}{ll}
   \mathbb{Q}& \text{if } n=m=0 \text{ and } *=0\\
 & \text{or }  n=1 \text{ and } *=m-1\\
 0 & \text{otherwise}
 \end{array}
 \right.$$
 
 $$Ext^{*}_{\F(\gr)}((\Lambda^n \circ \abe) \otimes \mathbb{Q}, (T^m \circ \abe) \otimes \mathbb{Q}) \simeq \left\lbrace\begin{array}{ll}
 \mathbb{Q}^{S(m,n)} & \text{if } *=m-n\\
 0 & \text{otherwise}
 \end{array}
 \right.$$
  
 $$Ext^{*}_{\F(\gr)}((S^n \circ \abe) \otimes \mathbb{Q}, (T^m \circ \abe) \otimes \mathbb{Q}) \simeq \left\lbrace\begin{array}{ll}
 \mathbb{Q}^{S(m,n)} & \text{if } *=m-n\\
 0 & \text{otherwise}
 \end{array}
 \right.$$
 where $S(m,n)$ denotes the Stirling partition number (i.e. the number of ways to partition a set of $m$ elements into $n$ non-empty subsets).

 $$Ext^{*}_{\F(\gr)}((T^n \circ \abe) \otimes \mathbb{Q}, (S^m \circ \abe) \otimes \mathbb{Q}) \simeq \left\lbrace\begin{array}{ll}
 \mathbb{Q}^{r(m,n)} & \text{if } m=n \text{\ and }*=0\\
 0 & \text{otherwise}
 \end{array}
 \right.$$
 
 and 
  $$Ext^{*}_{\F(\gr)}((T^n \circ \abe) \otimes \mathbb{Q}, (\Lambda^m \circ \abe) \otimes \mathbb{Q}) \simeq \left\lbrace\begin{array}{ll}
 \mathbb{Q}^{r(m,n)} & \text{if } *=m-n \\
 0 & \text{otherwise}
 \end{array}
 \right.$$
 where $r(m,n)$ denotes the number of ordered partition of $m$ into $n$ parts.

 \end{thm}

 The proof of this theorem relies on the following lemma:
 \begin{lm} \label{lm-tech}
 Let $G$ be a group and $E$ a $G$-set. Consider a $\mathbb{Q}$-linear  action of $G$ on $\mathbb{Q}[E]$, such that 
 $$g. [e]=\alpha(g,e)[g.e]$$
 where $\alpha: G \times E \to \{1, -1 \}$ is a map, then
 $$\mathbb{Q}[E]_G\simeq \mathbb{Q}[\{e \in E \mid \forall g \in G \textrm{ s.t. } g.e=e \textrm{ we have } \alpha(g,e)=1 \}/G].$$
 \end{lm}
 
 \begin{proof}[Proof of Theorem \ref{Ext-ext}]
 For $F$ and $G$ functors of type $T^k$, $S^k$ or $\Lambda^k$ we determine the action of $\mathfrak{S}_n \times \mathfrak{S}_m$  on $Ext^{*}_{\F(\gr)}((T^n \circ \abe) \otimes \mathbb{Q}, (T^m \circ \abe) \otimes \mathbb{Q})$ such that 
 $$Ext^{*}_{\F(\gr)}((F \circ \abe) \otimes \mathbb{Q}, (G \circ \abe) \otimes \mathbb{Q}) \simeq Ext^{*}_{\F(\gr)}((T^n \circ \abe) \otimes \mathbb{Q}, (T^m \circ \abe) \otimes \mathbb{Q}) _{\mathfrak{S}_n \times \mathfrak{S}_m}$$
$$\simeq \mathbb{Q}[Surj(m,n)]_{\mathfrak{S}_n \times \mathfrak{S}_m}.$$
By Lemma \ref{lm-tech}, it is sufficient to determine this action for the elements $(\sigma', \sigma'')\in {\mathfrak{S}_n \times \mathfrak{S}_m}$ such that $(\sigma', \sigma'').f=f$.
 
 Let $f$ be an element of $Surj(m,n)$ such that $f=s \circ \sigma$ is the unique decomposition of $f$ of the form described in Lemma \ref{map-Omega}. The group $\mathfrak{S}_n \times \mathfrak{S}_m$ acts on $Surj(m,n)$ by $(\sigma', \sigma'').f=\sigma' \circ f \circ \sigma''^{-1}$. The elements $(\sigma', \sigma'') \in \mathfrak{S}_n \times \mathfrak{S}_m$ such that $(\sigma', \sigma'').f=f$ are products of elements of the two following forms
 
 \begin{itemize}
\item  $(\tau_{i,j}, \sigma^{-1} T_{i,j} \sigma)$ for $i$ and $j$ such that $|f^{-1}(i)|=|f^{-1}(j)|$ and $T_{i,j}$ is the permutation obtained by the bloc transposition of $s^{-1}(i)$ and $s^{-1}(j)$;
\item $(Id, \sigma_{i_1} \ldots \sigma_{i_n} )$ where $ \sigma_{i_k}$ is a permutation of $f^{-1}(k)$.
\end{itemize}

In the following, for $F$ and $G$ functors of type $T^k$, $S^k$ or $\Lambda^k$ and $(\sigma', \sigma'') \in \mathfrak{S}_n \times \mathfrak{S}_m$ of the previous forms, we determine the map $\alpha: (\mathfrak{S}_n \times \mathfrak{S}_m) \times Surj(m,n) \to \{1, -1 \}$ such that $g. [e]=\alpha(g,e)[g.e]$.

 - For $F=\Lambda^n$ and $G=\Lambda^m$, we have:
$$\alpha((\tau_{i,j}, \sigma^{-1} T_{i,j} \sigma),f)= \epsilon(\tau_{i,j}) \epsilon(\sigma^{-1} T_{i,j} \sigma) \underset{1\leq i\leq n}{\Pi}\epsilon(\overline{\sigma^{-1} T_{i,j} \sigma_{\mid(f\circ (\sigma^{-1} T_{i,j} \sigma)^{-1}(i)}}) (-1)^{(|f^{-1}(i)|-1)^2}=1$$
since
$$\epsilon(\sigma^{-1} T_{i,j} \sigma) =\epsilon(T_{i,j}) =(-1)^{|f^{-1}(i)|^2}$$
and
$$\epsilon(\overline{\sigma^{-1} T_{i,j} \sigma_{\mid(f\circ (\sigma^{-1} T_{i,j} \sigma)^{-1}(i)}})=1;$$
and
$$\alpha((Id, \sigma_{i_1} \ldots \sigma_{i_n} ),f)= \epsilon( \sigma_{i_1} \ldots \sigma_{i_n})  \underset{1\leq i\leq n}{\Pi}\epsilon(\overline{{\sigma}_{i_1} \ldots {\sigma_{i_n}}_{\mid(f\circ (\sigma_{i_1} \ldots \sigma_{i_n} )^{-1}(i)}}) =1$$
since
$$\epsilon(\sigma_{i_1} \ldots \sigma_{i_n})  = \epsilon(\sigma_{i_1}) \ldots \epsilon(\sigma_{i_n}) $$
and
$$\epsilon(\overline{\sigma_{i_1} \ldots {\sigma_{i_n}}_{\mid(f\circ (\sigma_{i_1} \ldots \sigma_{i_n} )^{-1}(k)}})=\epsilon(\sigma_{i_k}).$$

By Lemma \ref{lm-tech} we deduce that
$$Ext^{*}_{\F(\gr)}((\Lambda^n \circ \abe) \otimes \mathbb{Q}, (\Lambda^m \circ \abe) \otimes \mathbb{Q})  \simeq \mathbb{Q}[Surj(m,n)]_{\mathfrak{S}_n \times \mathfrak{S}_m} \simeq \mathbb{Q}[Surj(m,n)/{\mathfrak{S}_n \times \mathfrak{S}_m}] $$
$$\simeq \left\lbrace\begin{array}{ll}
 \mathbb{Q}^{\rho(m,n)} & \text{if } *=m-n\\
 0 & \text{otherwise}
 \end{array}
 \right.$$
- For $F=S^n$ and $G=S^m$, 
we have $$\alpha((\tau_{i,j}, \sigma^{-1} T_{i,j} \sigma),f)=(-1)^{|f^{-1}(i)|^2+1},$$
$$\alpha((Id, \sigma_{i_1} \ldots \sigma_{i_n}),f)= \epsilon(\sigma_{i_1}) \ldots \epsilon(\sigma_{i_n}).$$
So, if $n \neq m$ 
$$\{f \in Surj(m,n)\mid \forall (\sigma', \sigma'')\in {\mathfrak{S}_n \times \mathfrak{S}_m} \textrm{ s.t. } (\sigma', \sigma'').f=f \textrm{ we have } \alpha((\sigma', \sigma''),f)=1 \}=\emptyset$$
and if $n=m$
$$\{f \in Surj(m,n)\mid \forall (\sigma', \sigma'')\in {\mathfrak{S}_n \times \mathfrak{S}_m} \textrm{ s.t. } (\sigma', \sigma'').f=f \textrm{ we have } \alpha((\sigma', \sigma''),f)=1 \}$$
$$=Surj(m,n).$$
- For $F=S^n$ and $G=\Lambda^m$,

$$\alpha((\tau_{i,j}, \sigma^{-1} T_{i,j} \sigma),f)= -1,$$
$$\alpha((Id, \sigma_{i_1} \ldots \sigma_{i_n} ),f)= 1.$$
So, if $n \neq 1$ or $n=0$ and $m \neq 0$
$$\{f \in Surj(m,n)\mid \forall (\sigma', \sigma'')\in {\mathfrak{S}_n \times \mathfrak{S}_m} \textrm{ s.t. } (\sigma', \sigma'').f=f \textrm{ we have } \alpha((\sigma', \sigma''),f)=1 \}=\emptyset$$
and if $n=1$ or $n=m=0$
$$\{f \in Surj(m,n)\mid \forall (\sigma', \sigma'')\in {\mathfrak{S}_n \times \mathfrak{S}_m} \textrm{ s.t. } (\sigma', \sigma'').f=f \textrm{ we have } \alpha((\sigma', \sigma''),f)=1 \}$$
is a set of cardinality $1$.

- For $F=\Lambda^n$ and $G=S^m$,
$$\alpha((\tau_{i,j}, \sigma^{-1} T_{i,j} \sigma),f)=(-1)^{|f^{-1}(i)|^2},$$
$$\alpha((Id, \sigma_{i_1} \ldots \sigma_{i_n}),f)= \epsilon(\sigma_{i_1}) \ldots \epsilon(\sigma_{i_n}).$$

- For $F=\Lambda^n$ and $G=T^m$, for $f \in Surj(m,n)$ and $\sigma' \in \mathfrak{S}_n$ we have $\sigma'.f \neq f$, so 
$$ \mathbb{Q}[Surj(m,n)]_{\mathfrak{S}_n}= \mathbb{Q}[Surj(m,n)/\mathfrak{S}_n]\simeq \mathbb{Q}^{S(m,n)}$$
 where $S(m,n)$ denotes the Stirling partition number.
 
 - For $F=S^n$ and $G=T^m$, for $f \in Surj(m,n)$ and $\sigma' \in \mathfrak{S}_n$ we have $\sigma'.f \neq f$, so 
$$ \mathbb{Q}[Surj(m,n)]_{\mathfrak{S}_n}= \mathbb{Q}[Surj(m,n)/\mathfrak{S}_n]\simeq \mathbb{Q}^{S(m,n)}.$$

 - For $F=T^n$ and $G=S^m$. We consider the action of $\mathfrak{S}_m$ on $Ext^{*}_{\F(\gr)}((T^n \circ \abe) \otimes \mathbb{Q}, (T^m \circ \abe) \otimes \mathbb{Q})$. We have:
 $$\alpha( \sigma_{i_1} \ldots \sigma_{i_n},f)= \epsilon(\sigma_{i_1}) \ldots \epsilon(\sigma_{i_n}).$$
So, if $n \neq m$ 
$$\{f \in Surj(m,n)\mid \forall \sigma''\in \mathfrak{S}_m \textrm{ s.t. }\sigma''.f=f \textrm{ we have } \alpha(\sigma'',f)=1 \}=\emptyset$$
and if $n=m$
$$\{f \in Surj(m,n)\mid \forall \sigma'' \in \mathfrak{S}_m\textrm{ s.t. } \sigma''.f=f \textrm{ we have } \alpha(\sigma'',f)=1 \}=Surj(m,n)$$
so 
$$ \mathbb{Q}[Surj(m,n)]_{\mathfrak{S}_m}= \mathbb{Q}[Surj(m,n)/\mathfrak{S}_m]\simeq \mathbb{Q}^{r(m,n)}$$
where $r(m,n)$ is the number of ordered partitions of $m$ into $n$ parts.

 - For $F=T^n$ and $G=\Lambda^m$. We have:
 $$\alpha( \sigma_{i_1} \ldots \sigma_{i_n},f)= (\epsilon(\sigma_{i_1}) \ldots \epsilon(\sigma_{i_n}))^2=1.$$
So,
$$\{f \in Surj(m,n)\mid \forall \sigma'' \in \mathfrak{S}_m\textrm{ s.t. } \sigma''.f=f \textrm{ we have } \alpha(\sigma'',f)=1 \}=Surj(m,n)$$
so 
$$ \mathbb{Q}[Surj(m,n)]_{\mathfrak{S}_m}= \mathbb{Q}[Surj(m,n)/\mathfrak{S}_m]\simeq \mathbb{Q}^{r(m,n)}$$
where $r(m,n)$ is the number of ordered partitions of $m$ into $n$ parts.

 \end{proof}

We deduce from Theorem \ref{Ext-ext} computations of rational $Tor$-groups. In fact, for $N: \textbf{gr}^{op} \to \mathbb{Q}\text{-}Mod$ we denote by $N^\vee$ the postcomposition of $N$ with the duality functor $V \mapsto Hom(V,\mathbb{Q})$. The $Tor$-groups can be deduced from the following natural graded isomorphism:
\begin{equation} \label{Ext-Tor}
Hom(Tor_\bullet^{\textbf{gr}}(N,M), \mathbb{Q}) \simeq Ext_{\mathcal{F}(\textbf{gr})}^\bullet(M,N^\vee)
\end{equation}
where $M: \textbf{gr} \to \mathbb{Q}\text{-}Mod$ 
(see \cite[Appendice A]{DV}).

For $N: \gr^{op} \to \mathbb{Q}\text{-}Mod$, $F_n$ the free group in $n$ generators and an integer $i$, Djament has shown \cite[Th\'eor\`eme $1.10$]{D15} that there is an isomorphism:
\begin{equation} \label{Djament-thm}
\underset{n \in \mathbb{N}}{\text{colim}}\ H_i(Aut(F_n), N(F_n)) \simeq \underset{k+l=i}{\bigoplus}{Tor}^\gr_k(N, \Lambda^l \circ(\mathfrak{a}\otimes \mathbb{Q})).
\end{equation}
As an application of Theorem \ref{Ext-ext}, using the previous result of Djament, we obtain the following computations of stable homology:

\begin{thm} 
Let $H_n=Hom_{\Gr}(F_n, \mathbb{Q})$.
We have
$$\underset{n \in \mathbb{N}}{\text{colim}}\ H_*(Aut(F_n), T^d\circ(\mathfrak{a}\otimes \mathbb{Q})(H_n)) \simeq  \left\lbrace\begin{array}{ll}
   \mathbb{Q}^{B(d)}& \text{if } *=d\\
 0 & \text{otherwise}
 \end{array}
 \right.$$
where $B(d)$ denotes the $d$th Bell number (i.e. the number of partitions of a set of $d$ elements);

$$\underset{n \in \mathbb{N}}{\text{colim}}\ H_*(Aut(F_n), \Lambda^d\circ(\mathfrak{a}\otimes \mathbb{Q})(H_n)) \simeq  \left\lbrace\begin{array}{ll}
   \mathbb{Q}^{\rho(d)}& \text{if } *=d\\
 0 & \text{otherwise}
 \end{array}
 \right.$$
where $\rho(d)$ denotes the number of partitions of $d$, and
$$\underset{n \in \mathbb{N}}{\text{colim}}\ H_*(Aut(F_n), S^d\circ(\mathfrak{a}\otimes \mathbb{Q})(H_n)) \simeq  \left\lbrace\begin{array}{ll}
   \mathbb{Q}& \text{if } *=d=0 \\
    & \text{or } *=d=1 \\
 0 & \text{otherwise}
 \end{array}
 \right.$$
where $S^d$ is the $d$-th symmetric power.
\end{thm}
\begin{proof}
For $F: \gr \to \mathbb{Q}\text{-Mod}$ such that for all $k \in \mathbb{N}$ $$Ext_{\mathcal{F}(\textbf{gr})}^k(\Lambda^l \circ(\mathfrak{a}\otimes \mathbb{Q}),F\circ(\mathfrak{a}\otimes \mathbb{Q}))$$ is a finite vector space, we have the isomorphisms
$$\underset{n \in \mathbb{N}}{\text{colim}}\ H_i(Aut(F_n), F\circ(\mathfrak{a}\otimes \mathbb{Q})(H_n)) \simeq  \underset{k+l=i}{\bigoplus}{Tor}^\gr_k(F\circ(\mathfrak{a}\otimes \mathbb{Q})\circ Hom_{\Gr}(-, \mathbb{Q}), \Lambda^l \circ(\mathfrak{a}\otimes \mathbb{Q}))
$$
$$\simeq  \underset{k+l=i}{\bigoplus}Hom(Ext_{\mathcal{F}(\textbf{gr})}^k(\Lambda^l \circ(\mathfrak{a}\otimes \mathbb{Q}),F\circ(\mathfrak{a}\otimes \mathbb{Q})), \mathbb{Q})$$
where the first isomorphism is Djament's isomorphism (\ref{Djament-thm}) and the second is the Tor-Ext isomorphism (\ref{Ext-Tor}).

For $F=T^d$, by Theorem \ref{Ext-ext} we obtain:
$$\underset{n \in \mathbb{N}}{\text{colim}}\ H_i(Aut(F_n), T^d\circ(\mathfrak{a}\otimes \mathbb{Q})(H_n)) \simeq
\left\lbrace\begin{array}{ll}
 0 & \text{if } i \neq d\\
  \underset{k+l=d}{\bigoplus} \mathbb{Q}^{S(d,l)}=\overset{d}{\underset{l=0}{\bigoplus}} \mathbb{Q}^{S(d,l)}=\mathbb{Q}^{B(d)}& \text{if } i=d\\
 \end{array}
 \right.$$
The other results are obtained in a similar way.
\end{proof}

\bibliographystyle{amsplain}
\bibliography{biblio-ext}

\end{document}